\documentclass[reqno]{amsart}
\usepackage{amssymb}
\usepackage{color}
\usepackage{mathrsfs}
\usepackage{amsmath, amsfonts, vmargin, enumerate}
\usepackage{graphics}
\usepackage{verbatim}
\usepackage{amsthm}
\usepackage{latexsym,bm}
\usepackage{euscript,amscd}
\usepackage{dashrule}
\usepackage{arydshln}
\usepackage{eufrak}
\usepackage{dsfont}

\input xypic
\xyoption {all}

\def\>{\rangle}
\def\<{\langle}

\newtheorem{theorem}{Theorem}[section]
\newtheorem{prop}[theorem]{Proposition}
\newtheorem{corollary}[theorem]{Corollary}

\theoremstyle{definition}
\newtheorem{defi}[theorem]{Definition}
\newtheorem{example}[theorem]{Example}

\theoremstyle{remark}
\newtheorem{remark}[theorem]{Remark}
\newtheorem{pb}[theorem]{Problem}

\numberwithin{equation}{section}

\newcommand{\bN}{{\mathbb N}}

\newcommand{\bM}{{\mathbb M}}
\newcommand{\bC}{{\mathbb C}}

\newcommand{\Tr}{\mathrm{Tr}}

\newcommand{\ba}{\begin{array}}
\newcommand{\ea}{\end{array}}
\newcommand{\be}{\begin{eqnarray*}}
\newcommand{\ee}{\end{eqnarray*}}
\newcommand{\beg}{\begin{eqnarray}}
\newcommand{\eeg}{\end{eqnarray}}
\newcommand{\beq}{\begin{equation}}
\newcommand{\eeq}{\end{equation}}
\newcommand{\beqn}{\begin{equation*}}
\newcommand{\eeqn}{\end{equation*}}

\renewcommand{\Re}{\mathrm{Re}}

\newcommand*\oline[1]{%
  \vbox{%
    \hrule height 0.5pt
    \kern0.25ex
    \hbox{%
      \kern-0.1em
      \ifmmode#1\else\ensuremath{#1}\fi
      \kern-0.1em
    }
  }
}
\newcommand{\id}{\mathrm{id}}

\newcommand{\MM}{\mathbb{M}_3(\mathbb{C})}

\title{Generalizing Choi-like maps}
\author{Dariusz Chru\'{s}ci\'{n}ski}
\address{Institute of Physics, Nicolaus Copernicus University\\
	Grudzi\c{a}dzka 5/7, 87--100 Toru\'{n}, Poland}
\email{darch@fizyka.umk.pl}
\author{Marcin Marciniak}
\address{Institute of Theoretical Physics and Astrophysics \\ Faculty of Mathematics, Physics and Informatics \\ University of Gda{\'n}sk \\ 80-309 Gda{\'n}sk, Poland}
\email{matmm@ug.edu.pl}
\author{Adam Rutkowski}
\address{Institute of Theoretical Physics and Astrophysics, Faculty of Mathematics, Physics and Informatics, University of Gda{\'n}sk, 80-309 Gda{\'n}sk, Poland}
\email{fizar@ug.edu.pl}

\begin{document}
	\maketitle
	
\begin{abstract}
A problem of further generalization of generalized Choi maps $\Phi_{[a,b,c]}$ acting on $\bM_3$ introduced by Cho, Kye and Lee is discussed. Some necessary conditions for positivity of the generalized maps are provided as well as some sufficient conditions. Also some sufficient condition for decomposability of these maps is shown.
\end{abstract}

\section{Introduction}
The paper concerns linear maps acting between matrix algebras. 
For $n\in\bN$, by $\bM_n$ we denote the algebra of square $n\times n$-matrices $X=(x_{ij})$ with complex coefficients. Let $\Phi:\bM_m\to\bM_n$ be a linear map. We say that $\Phi$ is a \textit{positive map} if $\Phi(X)$ is a positive definite matrix in $\bM_n$ for every positive definite matrix $X\in \bM_m$. Remind that for $k\in\bN$ there are the following identifications: $\bM_k(\bM_m)=\bM_k\otimes\bM_m=\bM_{km}$, where $\bM_k(\bM_m)$ denotes the C*-algebra of $k\times k$-matrices $(X_{ij})$ with $X_{ij}\in \bM_m$. Define a map $\Phi_k:\bM_k(\bM_m)\to\bM_k(\bM_n)$ by $\Phi_k(X_{ij})=(\Phi(X_{ij}))$. We say that the map $\Phi$ is \textit{$k$-positive} if $\Phi_k$ is a positive map. If $\Phi$ is $k$-positive for every $k\in\bN$ then it is called a \textit{completely positive map}.

The transposition map $\theta_n:\bM_n\to\bM_n:(x_{ij})\mapsto (x_{ji})$ is known to be not completely positive map (even not $2$-positive) \cite{Sto12}. We say that a map $\Phi:\bM_n\to\bM_n$ is \textit{$k$-copositive} (respectively \textit{completely copositive}) if the composition $\theta_n\circ\Phi$ is $k$-positive (respectively completely positive). A map $\Phi$ is said to be \textit{decomposable} if it can be expressed as a sum $\Phi=\Phi_1+\Phi_2$ where $\Phi_1$ is a completely positive map while $\Phi_2$ is a completely copositive one. Otherwise $\Phi$ is called \textit{indecomposable}. For further details we refer the reader to the book of St{\o}rmer \cite{Sto12}.

The systematic study of positive maps on C*-algebras was started by a pioneering work of Erling St{\o}rmer in the 60's of the last century \cite{Sto63}. Although the definitions are easy and the above notions seem to be rather elementary, the problem of description of all positive maps is still unsolved even in the case of low dimensional matrix algebras. In particular, still there is no effective characterization of decomposable maps. In recent years the importance of the theory of positive maps drastically increased because of its applications in mathematical physics, especially in rapidly emerging theory of quantum information \cite{HHH96,MajMar01,ChruSarb2014}.
 
It was shown by St{\o}rmer \cite{Sto63} and Woronowicz \cite{Wor76a} that every positive map $\Phi:\bM_2\to\bM_2$ (St{\o}rmer) or $\Phi:\bM_2\to\bM_3$ (Woronowicz) is decomposable. The first example of an indecomposable positive map was given by Choi \cite{Choi75b,ChoiLam77}. It is a map $\Phi:\bM_3\to\bM_3$ given by
\begin{equation}
\Phi(X)=\left(\begin{array}{ccc}
x_{11}+x_{33} & -x_{12} & -x_{13}\\
-x_{21} & x_{22}+x_{11} & -x_{23}\\
-x_{31} & -x_{32} & x_{33}+x_{22}
\end{array}\right).
\label{Choi}
\end{equation}
Although there is no systematic prescription of indecomposable maps, there are numerous examples scattered across literature \cite{Robertson83,Osaka91,Osaka92,Osaka93,Tang86,Ha98,Ha02,Ha03,TanTom88,ChruSarb2012b,MilOlk15,MarRut17}.

An interesting approach to generalization of \eqref{Choi} was presented in \cite{CKL92}.
For any triple of nonnegative numbers $a,b,c$ define a map $\Phi_{[a,b,c]}:\bM_3\to\bM_3$ by	
	\begin{equation}
	\Phi_{[a,b,c]}(X)=\left(\begin{array}{ccc}
	ax_{11}+bx_{22}+cx_{33} & -x_{12} & -x_{13}\\
	-x_{21} & cx_{11}+ax_{22}+bx_{33} & -x_{23}\\
	-x_{31} & -x_{32} & bx_{11}+cx_{22}+ax_{33}
	\end{array}\right)\ ,\label{}
	\end{equation}
	with $x_{ij}$ being the matrix elements of $X\in\MM$. One proves
\begin{theorem}[\cite{CKL92}] 
\label{TH-korea} 
The map $\Phi_{[a,b,c]}$ is positive but not completely positive if and only if 
\begin{enumerate}
\item $0\leq a<2\ $, 
\item $a+b+c\geq2\ $, 
\item if $a\leq1\ $, then $\ bc\geq(1-a)^{2}$. 
\end{enumerate}
Moreover, being positive it is indecomposable if and only if $4bc<(2-a)^{2}$.
\end{theorem} 
Slightly different class was considered by Kye \cite{Kye92}
	\begin{equation}
	\Phi_{[a;c_{1},c_{2},c_{3}]}[X]=\left(\begin{array}{ccc}
	ax_{11}+c_{1}x_{33} & -x_{12} & -x_{13}\\
	-x_{21} & c_{2}x_{11}+ax_{22} & -x_{23}\\
	-x_{31} & -x_{32} & c_{3}x_{22}+ax_{33}
	\end{array}\right)\ ,\label{}
	\end{equation}
	with $a,c_{1},c_{2},c_{3}\geq0$. He proved the following
	\begin{theorem} \label{TH-ccc} 
A map $\Phi_{[a;c_{1},c_{2},c_{3}]}$
		is positive but not completely positive if and only if 
		\begin{enumerate}
			\item $1\leq a<2\ $, 
			\item $c_{1}c_{2}c_{3}\geq(2-a)^{3}$. 
		\end{enumerate}
		Moreover, $\Phi_{[a;c_{1},c_{2},c_{3}]}$ is indecomposable and atomic.
	\end{theorem} 
Interestingly, Osaka shown 
\begin{theorem}[\cite{Osaka92}] \label{TH-Osaka} 
A map $\Phi_{[1;c_{1},c_{2},c_{3}]}$ is extremal if $c_{1}c_{2}c_{3}=1$. \end{theorem}
	
Let us remind that for a linear map $\Phi:\bM_m\to\bM_n$ one can define its Choi matrix \cite{Choi75}. It is an element $C_\Phi\in\bM_m(\bM_n)$ defined by
\begin{equation}
\label{e:ChoiM}
C_\Phi=(\Phi(E_{ij}))
\end{equation}
where $\{E_{ij}:\,1\leq i,j\leq n\}$ is a system of matrix units in $\bM_m$. The mapping $\Phi\mapsto C_\Phi$ is the so called Choi-Jamio{\l}kowski isomorphism between $L(\bM_m,\bM_n)$ and $\bM_m(\bM_n)$ \cite{Jam74}. One has
\begin{theorem}[\cite{Choi75}]
\label{t:Choi}
A map $\Phi$ is completely positive if and only if its Choi matrix $C_\Phi$ is positive definite.
\end{theorem}
It turns out also that it is possible to characterise positivity and decomposability of $\Phi$ in terms of the matrix $C_\Phi$. To this end let us recall that $C_\Phi$ can be considered as an element of the tensor product $\bM_m\otimes\bM_n$. An element $\rho\in M_m\otimes M_n$ is called a PPT matrix if both $\rho$ and $(\id_{\bM_m}\otimes\theta_n)(\rho)$ are positive definite.
\begin{theorem}[\cite{HHH96,MajMar01,Sto82}]
\label{t:dec}
Let $\Phi$ be a linear map. 
\begin{enumerate}
	\item[(i)]
	$\Phi$ is positive if and only if for every vectors $\xi\in\bC^m$, $\eta\in\bC^n$
	\begin{equation}
	\label{e:block}
	\<\xi\otimes\eta,C_\Phi\xi\otimes\eta\>\geq 0
	\end{equation}
	for every vectors $\xi\in\bC^m$, $\eta\in\bC^n$.
	\item[(ii)]
	$\Phi$ is decomposable if and only if 
	\begin{equation}
	\Tr(\rho C_\Phi)\geq 0
	\end{equation}
	for every PPT matrix $\rho\in\bM_m\otimes\bM_n$.
\end{enumerate}
\end{theorem}
The property (i) in the above theorem is called \textit{block-positivity} \cite{MajMar01}.

\section{Choi-like maps}

In this paper we consider the following generalization of the maps described in Introduction. Let 
\begin{equation}
	A=\left(\begin{array}{ccc}
	a_{1} & b_{1} & c_{1}\\
	c_{2} & a_{2} & b_{2}\\
	b_{3} & c_{3} & a_{3}
	\end{array}\right),\label{matrA}
	\end{equation}
	with $a_{i},b_{j},c_{k}\geq0$, and define $\Phi_{A}:\bM_3\to\bM_3$
	as follows 
	\begin{equation}
	\Phi_{A}(X)=\left(\begin{array}{ccc}
	a_{1}x_{11}+b_{1}x_{22}+c_{1}x_{33} & -x_{12} & -x_{13}\\
	-x_{21} & c_{2}x_{11}+a_{2}x_{22}+b_{2}x_{33} & -x_{23}\\
	-x_{31} & -x_{32} & b_{3}x_{11}+c_{3}x_{22}+a_{3}x_{33}
	\end{array}\right)\ .\label{PhiA}
	\end{equation}
Our aim is to describe some conditions for positivity of $\Phi_A$ as well as for decomposability.

Firstly, let us make the following observation.
\begin{prop}
A map $\Phi_{A}$ is completely positive iff 
	\begin{equation}
	\left(\begin{array}{ccc}
	a_{1} & -1 & -1\\
	-1 & a_{2} & -1\\
	-1 & -1 & a_{3}
	\end{array}\right)\label{e:ACP}
	\end{equation}
is positive definite matrix.
\end{prop}
\begin{proof}
	Observe that the Choi matrix of the map $\Phi_A$ is of the form
	\begin{equation}
C_{\Phi_A}=	\left(
\begin{array}{ccc|ccc|ccc}
a_1 & \cdot & \cdot & \cdot & -1 & \cdot & \cdot & \cdot & -1 \\
\cdot & b_1 & \cdot & \cdot & \cdot & \cdot & \cdot & \cdot & \cdot \\
\cdot & \cdot & c_1 & \cdot & \cdot & \cdot & \cdot & \cdot & \cdot \\ \hline
\cdot & \cdot & \cdot & c_2 & \cdot & \cdot & \cdot & \cdot & \cdot \\
-1 & \cdot & \cdot & \cdot & a_2 & \cdot & \cdot & \cdot & -1 \\
\cdot & \cdot & \cdot & \cdot & \cdot & b_2 & \cdot & \cdot & \cdot \\ \hline
\cdot & \cdot & \cdot & \cdot & \cdot & \cdot & b_3 & \cdot & \cdot \\
\cdot & \cdot & \cdot & \cdot & \cdot & \cdot & \cdot & c_3 & \cdot \\
-1 & \cdot & \cdot & \cdot & -1 & \cdot & \cdot & \cdot & a_3
\end{array}
\right)
\end{equation}
where we used dots instead of zeros. Now we apply Theorem \ref{t:Choi}. Since $b_j,c_k$ are nonnegative, positive definiteness of $C_{\Phi_A}$ is equivalent  to positive definiteness of the submatrix  \eqref{e:ACP}.
\end{proof}
	
Now, let us discuss some necessary conditions for positivity.	
\begin{theorem} 
A necessary condition for positivity of $\Phi_{A}$ is positivity of $\Phi_{[\overline{a},\overline{b},\overline{c}]}$, where
\begin{eqnarray}
		\overline{a}=\frac{1}{3}(a_{1}+a_{2}+a_{3})\ ,\ \overline{b}=\frac{1}{3}(b_{1}+b_{2}+b_{3})\ ,\ \overline{c}=\frac{1}{3}(c_{1}+c_{2}+c_{3}).\label{abc}
\end{eqnarray}
\end{theorem} 
\begin{proof} 
	Let $\Phi_{A}$ be a positive map and consider
	a linear map 
	\begin{equation}
\overline{\Phi}_{A}(X)=\frac{1}{3}\Big(\Phi_{A}(X)+S\Phi_{A}(S^{*}XS)S^{*}+S^{*}\Phi_{A}(SXS^{*})S\Big)\ ,\label{}
	\end{equation}
	where $S$ denotes a unitary shift $Se_{i}=e_{i+1}$. It is clear
	that $\overline{\Phi}_{A}$ is positive. Moreover, $\overline{\Phi}_{A}=\Phi[\overline{a},\overline{b},\overline{c}]$
	with $\overline{a},\overline{b},\overline{c}$ defined in (\ref{abc}).
\end{proof}
	
Let us introduce the following notation
\begin{eqnarray}
a_*=(a_{1}a_{2}a_{3})^{1/3}\ ,\ b_*=(b_{1}b_{2}b_{3})^{1/3}\ ,\ c_*=(c_{1}c_{2}c_{3})^{1/3}.\label{gabc}
\end{eqnarray}
To find sufficient conditions consider a positive map $\Phi[a,b,c]$ and define 
	\begin{equation}
	\Phi_{[a,b,c]}^{V}(X):=V^{1/2}\Phi_{[a,b,c]}\left(V^{-1/2}XV^{-1/2}\right)V^{1/2},\label{}
	\end{equation}
	where $V=\sum_{i=1}^{3}p_{i}E_{ii}$ with $p_{i}>0$. It is clear
	that $\Phi_{[a,b,c]}^{V}$ is positive. Moreover, one easily finds
	that $\Phi_{[a,b,c]}^{V}=\Phi_{A}$ with 
	\begin{equation}
	a_{i}=a_{*}\ ,\ b_{i}=\frac{p_{i+1}}{p_{i}}b\ ,\ c_{i}=\frac{p_{i+2}}{p_{i}}c\ ,\ \ \ i=1,2,3\ ,\label{}
	\end{equation}
	that is, $A=V^{-1}A_{[a,b,c]}V$. Note that $b_*=b$,
	$c_*=c$, and $b_{i}c_{i+1}=bc$. Hence, one arrives
	at the following
	\begin{theorem} 
		Assume that there exist
		$a,b,c$ and $V$ such that $\Phi_{[a,b,c]}$ is a positive map and
		the matrix 
		\begin{equation}
		M=A-V^{-1}A_{[a,b,c]}V,\label{M}
		\end{equation}
		has non-negative elements. Then
		$\Phi_{A}$ is a positive map.  
	\end{theorem}
	
	\begin{corollary} 
		Assume that there exist $a,b,c\geq0$
		such that 
		\begin{enumerate}
			\item 
		$\Phi_{[a,b,c]}$ defines a positive map, 
\item		
		$\min\{a_{1},a_{2},a_{3}\}\geq a$, \quad
		$b_*\geq b$, \quad $c_*\geq c$, 
\item 		$b_{i}c_{i+1}\geq bc$ for $i=1,2,3$.
\end{enumerate}
Then  $\Phi_{A}$ is a positive map.
  \end{corollary}
	
	\begin{remark} If $b_{1}b_{2}b_{3}=0$, then a sufficient condition
		reduces to 
		\begin{enumerate}
			\item $a_{1},a_{2},a_{3}\geq a$,
			\item $c_{1}c_{2}c_{3}\geq(2-a)^{3}$,
		\end{enumerate}
		with $1\leq a<2$. Interestingly, if $a_{1}=a_{2}=a_{3}=a$, then
		$a\geq1$ and $c_{1}c_{2}c_{3}\geq(2-a)^{3}$ are necessary and sufficient
		\cite{Kye92}. Indeed, if $\Phi_{A}$ is positive and $c_{1}c_{2}c_{3}=c^{3}$
		let us define $V$ by taking 
		\begin{equation}
		c_{1}=\frac{p_{3}}{p_{1}}c\ ,\ c_{2}=\frac{p_{1}}{p_{2}}c\ ,\ c_{3}=\frac{p_{2}}{p_{3}}c\ \ \label{}
		\end{equation}
		Note the map $\Phi_{[a,b,c]}$ has to be positive and hence $a+c\geq0$.
		Hence $a+\sqrt[3]{c_{1}c_{2}c_{3}}\geq2$ which implies $c_{1}c_{2}c_{3}\geq(2-a)^{3}$.
	\end{remark}
	
	We end this section with the following remark.
\begin{remark}
\label{u:1}
Let us observe that condition (3) in Theorem \ref{TH-korea}	can equivalently written as
\begin{equation}
a+\sqrt{bc}\geq 1.
\end{equation}
On the other hand, the condition (2) of Theorem \ref{TH-ccc} can be rewritten as
\begin{equation}
a+c_*\geq 2.
\end{equation}
Let us also remind that for a matrix $$A=\left(\begin{array}{cc}a_1 & b_1 \\b_2 & a_2\end{array}\right)$$ one can define a map $\Phi_A:\bM_2\to\bM_2$ similarly to \eqref{PhiA}. It was shown by Kossakowski (see \cite[Example 7.1]{ChruSarb2014})
that $\Phi_A$ is positive if and only if $a_{i}\geq0$, $b_{i}\geq0$
and 
\begin{equation}
\sqrt{a_{1}a_{2}}+\sqrt{b_{1}b_{2}}\geq1 \label{aibi2}
\end{equation}
Having all these observations in mind one can conjecture that positivity of the map \eqref{PhiA} is related to the following set of conditions:
\begin{eqnarray}
& \sqrt{a_ia_{i+1}}+\sqrt{b_ic_{i+1}}\geq 1& \label{c:2} \\
& a_*+b_*+c_*\geq 2 &  \label{c:3} 
\end{eqnarray}
In next sections we shall discuss these conditions.
\end{remark}
\section{Condition \eqref{c:2}}
Now we consider further generalisation.	
	Let $A$=$\left(a_{ij}\right)_{i,j=1}^{n}$ be a matrix with nonnegative coefficients. 
	Let us define a map  $\Phi_A:\bM_n\to\bM_n$ by
	\begin{equation}
	\Phi_{A}\left(X\right)=\Delta_{A}\left(X\right)-X,
\label{PhiAn}
	\end{equation}
	where
	\begin{equation}
	\Delta_{A}\left(X\right)=\textrm{diag}\left(\sum_{j}a_{1j}'x_{jj},\sum_{j}a_{2j}'x_{jj},\ldots,\sum_{j}a_{nj}'x_{jj}\right).
	\end{equation}
In the above formula $a_{ij}'=a_{ij}+\delta_{ij}$, where $\delta_{ij}$ stands for the Kronecker delta.	
	
	%
	%
	%
	%
	%
	%

	Observe that positivity of $\Phi_A$ is equivalent to the inequality
$\left\langle \eta,\Delta_{A}\left(\zeta\zeta^{*}\right)\eta\right\rangle \geq\left\langle \eta,\zeta\zeta^{*}\eta\right\rangle$ 
for every $\zeta,\eta\in\mathbb{C}^n$.	
	One can easily show that the above is equivalent to
	\begin{equation}
	\sum_{ij}a_{ij}'p_{i}^{2}q_{j}^{2}\geq\left(\sum_{i}p_{i}q_{i}\right)^{2} \label{eq:posit},\qquad p_{i},q_{i}\geq0
	\end{equation}

In the following main result of this section we discuss the role of the  condition \eqref{c:2}. 
\begin{theorem}
Let $\Phi_A$ be given by \eqref{PhiAn}. 
If the map $\Phi_A$ is positive, then 
\begin{equation}
\label{local}
\sqrt{a_{ii}a_{jj}}+\sqrt{a_{ij}a_{ji}}\geq1,\qquad i\neq j
\end{equation}
\end{theorem}
\begin{proof}
	Let $i\neq j$ be fixed. Observe that 
		\begin{eqnarray}
		\lefteqn{\sum_{kl}a_{kl}'p_{k}^{2}q_{l}^{2}-\left(\sum_{k}p_{k}q_{k}\right)^{2}=}\label{st}\\ 
		& = &
		\sum_{kl}a_{kl}'p_{k}^{2}q_{l}^{2}-\sum_{kl}p_{k}p_{l}q_{k}q_{l}  \\
		&=&
		\sum_{k} a_{kk} p_{k}^{2}q_{k}^{2}+\sum_{k<l}\left(a_{kl}p_{k}^{2}q_{l}^{2}+a_{lk}p_{l}^{2}q_{k}^{2}\right)-2\sum_{k<l}p_{k}p_{l}q_{k}q_{l} \\
		& = & 
		\sum_{k} a_{kk} p_{k}^{2}q_{k}^{2} + \sum_{k<l}(\sqrt{a_{kl}}p_kq_l-\sqrt{a_{lk}}p_lq_k)^2+2 \sum_{k<l}(\sqrt{a_{kl}a_{lk}}-1)p_kp_lq_kq_l
		\label{eq:posit2} 
		\end{eqnarray}
		Consider the case $p_k=0$ and $q_k=0$ for every $k\not\in\{i,j\}$. Then, the above expression is equal to
		\begin{eqnarray}
		\lefteqn{a_{ii}p_i^2q_i^2 + a_{jj}p_j^2q_j^2 +(\sqrt{a_{ij}}p_iq_j-\sqrt{a_{ji}}p_jq_i)^2+2(\sqrt{a_{ij}a_{ji}}-1)p_ip_jq_iq_j=} \nonumber\\
		&=&
		(\sqrt{a_{ii}}p_iq_i-\sqrt{a_{jj}}p_jq_j)^2 + (\sqrt{a_{ij}}p_iq_j-\sqrt{a_{ji}}p_jq_i)^2 \label{l1} \\ && {}+2\left(\sqrt{a_{ii}a_{jj}}+\sqrt{a_{ij}a_{ji}}-1\right)p_ip_jq_iq_j \label{l2}
		\end{eqnarray}
		Now, take
		$$
		p_i=q_j=1,\qquad p_j=\left(\dfrac{a_{ij}a_{ii}}{a_{ji}a_{jj}}\right)^{1/4}, \qquad q_i=\left(\dfrac{a_{ij}a_{jj}}{a_{ji}a_{ii}}\right)^{1/4}
		$$
		Then, the expression from lines \eqref{l1} and \eqref{l2} reduces to
		$$2\left(\sqrt{a_{ii}a_{jj}}+\sqrt{a_{ij}a_{ji}}-1\right)\left(\dfrac{a_{ij}}{a_{ji}}\right)^{1/2}$$
		It follows from the assumption that the expression \eqref{st} is nonnegative for every choice of $p_k$, $q_k$. Thus we arrive at the inequality
		$$\sqrt{a_{ii}a_{jj}}+\sqrt{a_{ij}a_{ji}}-1\geq 0$$
		which is equivalent to \eqref{local}.
\end{proof}

Having in mind Theorem \ref{TH-korea} one cannot expect that a converse theorem is also true. However, it turns out that a slight modification of \eqref{local} gives a sufficient condition for positivity and even decomposability. 
\begin{theorem}
If
\begin{equation}
\label{nlocal}
\dfrac{1}{n-1}\sqrt{a_{ii}a_{jj}}+\sqrt{a_{ij}a_{ji}}\geq1,\qquad i\neq j
\end{equation}
then the map $\Phi_A$ is positive and decomposable.
\end{theorem}
\begin{proof}	
		Let us observe that
		\begin{eqnarray}
		\lefteqn{\sum_{i} a_{ii} p_{i}^{2}q_{i}^{2}=}\\
		&=&\frac{1}{n-1} \sum_{i<j}\Big( 
		a_{ii}p_{i}^{2}q_{i}^{2}+a_{jj}p_{j}^{2}q_{j}^{2}
		\Big)
		\\
		&=&\sum_{i<j}\left(\sqrt{\frac{a_{ii}}{n-1}}p_{i}q_{i}-\sqrt{\frac{a_{jj}}{n-1}}p_{j}q_{j}\right)^{2}+2\sum_{i<j}\dfrac{\sqrt{a_{ii}a_{jj}}}{n-1}p_{i}p_{j}q_{i}q_{j}\label{yy}
		\end{eqnarray}
		According to the lines \eqref{st} - \eqref{eq:posit2} and \eqref{yy} one has
		\begin{eqnarray}
		\lefteqn{\sum_{ij}a_{ij}'p_{i}^{2}q_{j}^{2}-\left(\sum_{i}p_{i}q_{i}\right)^{2}=}\\ 
		& = & 
		\sum_{i}a_{ii}p_{i}^{2}q_{i}^{2} + \sum_{i<j}(\sqrt{a_{ij}}p_iq_j-\sqrt{a_{ji}}p_jq_i)^2+2 \sum_{i<j}(\sqrt{a_{ij}a_{ji}}-1)p_ip_jq_iq_j\\
		&=&
		\sum_{i<j}\left(\sqrt{\frac{a_{ii}}{n-1}}p_{i}q_{i}-\sqrt{\frac{a_{jj}}{n-1}}p_{j}q_{j}\right)^{2} + \sum_{i<j}(\sqrt{a_{ij}}p_iq_j-\sqrt{a_{ji}}p_jq_i)^2 \\
		&& {}+ 2\sum_{i<j}\left(\dfrac{\sqrt{a_{ii}a_{jj}}}{n-1}+\sqrt{a_{ij}a_{ji}}-1\right)p_ip_jq_iq_j
		\end{eqnarray}
		It follows from \eqref{nlocal} that the above expression is nonnegative for every nonnegative $p_i$ and $q_i$, so $\Phi_A$ is positive.

Now, we will show that $\Phi_A$ is decomposable. Let $\rho\in \bM_n(\bM_n)$ be a PPT matrix. We have $\rho=(\rho_{ij})$ where $\rho_{ij}\in\bM_n$. For $i\neq j$, let $r_{ij}$ be the $(i,j)$-coefficient of the matrix $\rho_{ij}$. For any $i,j$, let $\alpha_{ij}$ denote the $j$-th diagonal term in the matrix $\rho_{ii}$. Thus the matrix $\rho$ looks like
$$
\rho=\left(\begin{array}{ccccc|ccccc|c|ccccc}
\alpha_{11} & * & * &\cdots & * & * & r_{12} & * & \cdots & * & \cdots & * & * &* & \cdots & r_{1n} \\
* & \alpha_{12} & * &\cdots & * & * & * & * & \cdots & * & \cdots & * & * &* & \cdots & * \\
* & * & \alpha_{13} &\cdots & * & * & * & * & \cdots & * & \cdots & * & * &* & \cdots & * \\
\vdots & \vdots & \vdots & & \vdots & \vdots & \vdots & \vdots & & \vdots &  &\vdots & \vdots & \vdots & & \vdots \\
* & * & * &\cdots & \alpha_{1n} & * & * & * & \cdots & * & \cdots & * & * &* & \cdots & * \\ \hline
* & * & * &\cdots & * & \alpha_{21} & * & * & \cdots & * & \cdots & * & * &* & \cdots & * \\
r_{21} & * & * &\cdots & * & * & \alpha_{22} & * & \cdots & * & \cdots & * & * &* & \cdots & r_{2n} \\
* & * & * &\cdots & * & * & * & \alpha_{23} & \cdots & * & \cdots & * & * &* & \cdots & * \\
\vdots & \vdots & \vdots & & \vdots & \vdots & \vdots & \vdots & & \vdots &  &\vdots & \vdots & \vdots & & \vdots \\
* & * & * &\cdots & * & * & * & * & \cdots & \alpha_{2n} & \cdots & * & * &* & \cdots & * \\ \hline
\vdots & \vdots & \vdots & & \vdots & \vdots & \vdots & \vdots & & \vdots &  &\vdots & \vdots & \vdots & & \vdots \\\hline
* & * & * &\cdots & * & * & * & * & \cdots & * & \cdots & \alpha_{n1} & * &* & \cdots & * \\
* & * & * &\cdots & * & * & * & * & \cdots & * & \cdots & * & \alpha_{n2} & * & \cdots & * \\
* & * & * &\cdots & * & * & * & * & \cdots & * & \cdots & * & * & \alpha_{n3} & \cdots & * \\
\vdots & \vdots & \vdots & & \vdots & \vdots & \vdots & \vdots & & \vdots &  &\vdots & \vdots & \vdots & & \vdots \\
r_{n1} & * & * &\cdots & * & * & r_{n2} & * & \cdots & * & \cdots & * & * &* & \cdots & \alpha_{nn}
\end{array}\right)
$$
while
$$
\rho^\Gamma=\left(\begin{array}{ccccc|ccccc|c|ccccc}
\alpha_{11} & * & * &\cdots & * & * & * & * & \cdots & * & \cdots & * & * &* & \cdots & * \\
* & \alpha_{12} & * &\cdots & * & r_{12} & * & * & \cdots & * & \cdots & * & * &* & \cdots & * \\
* & * & \alpha_{13} &\cdots & * & * & * & * & \cdots & * & \cdots & * & * &* & \cdots & * \\
\vdots & \vdots & \vdots & & \vdots & \vdots & \vdots & \vdots & & \vdots &  &\vdots & \vdots & \vdots & & \vdots \\
* & * & * &\cdots & \alpha_{1n} & * & * & * & \cdots & * & \cdots & r_{1n} & * &* & \cdots & * \\ \hline
* & r_{21} & * &\cdots & * & \alpha_{21} & * & * & \cdots & * & \cdots & * & * &* & \cdots & * \\
* & * & * &\cdots & * & * & \alpha_{22} & * & \cdots & * & \cdots & * & * & * & \cdots & * \\
* & * & * &\cdots & * & * & * & \alpha_{23} & \cdots & * & \cdots & * & * & * & \cdots & * \\
\vdots & \vdots & \vdots & & \vdots & \vdots & \vdots & \vdots & & \vdots &  &\vdots & \vdots & \vdots & & \vdots \\
* & * & * &\cdots & * & * & * & * & \cdots & \alpha_{2n} & \cdots & * & r_{2n} &* & \cdots & * \\ \hline
\vdots & \vdots & \vdots & & \vdots & \vdots & \vdots & \vdots & & \vdots &  &\vdots & \vdots & \vdots & & \vdots \\\hline
* & * & * &\cdots & r_{n1} & * & * & * & \cdots & * & \cdots & \alpha_{n1} & * &* & \cdots & * \\
* & * & * &\cdots & * & * & * & * & \cdots & r_{n2} & \cdots & * & \alpha_{n2} & * & \cdots & * \\
* & * & * &\cdots & * & * & * & * & \cdots & * & \cdots & * & * & \alpha_{n3} & \cdots & * \\
\vdots & \vdots & \vdots & & \vdots & \vdots & \vdots & \vdots & & \vdots &  &\vdots & \vdots & \vdots & & \vdots \\
* & * & * &\cdots & * & * & * & * & \cdots & * & \cdots & * & * &* & \cdots & \alpha_{nn}
\end{array}\right)
$$
where $\rho^\Gamma= (\theta_n(\rho_{ij}))$ is the partially transposed $\rho$, and stars stand for any numbers.
Since $\rho$ and $\rho^\Gamma$ are positive definite, $\alpha_{ij}\geq 0$, $r_{ji}=\overline{r_{ij}}$, and
\begin{equation}
a_{ii}a_{jj}\geq |r_{ij}|^2,\qquad a_{ij}a_{ji}\geq |r_{ij}|^2,\qquad i\neq j.
\label{e:inrho}
\end{equation}
Now, observe that 
\begin{eqnarray}
\Tr(\rho C_{\Phi_A})&=&\sum_{i,j=1}^n a_{ij}\alpha_{ij}-2\sum_{1\leq i<j\leq n}\Re( r_{ij}) \nonumber \\
& \geq &
\sum_{i=1}^n a_{ii}\alpha_{ii} + \sum_{1\leq i<j\leq n}(a_{ij}\alpha_{ij}+a_{ji}\alpha_{ji})-2\sum_{1\leq i<j\leq n}|r_{ij}|\nonumber\\
&=&
\frac{1}{n-1}\sum_{1\leq i<j\leq n}(a_{ii}\alpha_{ii}+a_{jj}\alpha_{jj})+\sum_{1\leq i<j\leq n}(a_{ij}\alpha_{ij}+a_{ji}\alpha_{ji})-2\sum_{1\leq i<j\leq n}|r_{ij}| \nonumber \\ 
&\geq & 2\sum_{i<j}\left(\frac{1}{n-1}\sqrt{a_{ii}a_{jj}\alpha_{ii}\alpha_{jj}}+\sqrt{a_{ij}a_{ji}\alpha_{ij}\alpha_{ji}}-|r_{ij}|\right) \label{linia1}\\
&\geq &
2\sum_{i<j}|r_{ij}|\left(\frac{1}{n-1}\sqrt{a_{ii}a_{jj}}+\sqrt{a_{ij}a_{ji}}-1\right)\geq 0.\label{linia2}
\end{eqnarray}
The expression \eqref{linia1} was obtained by arithmetic-geometric mean inequality, while \eqref{linia2} is due to inequalities \eqref{e:inrho}.
Since $\rho$ is an arbitrary PPT matrix, we conclude that $\Phi_A$ is decomposable (cf. Theorem \ref{t:dec}(ii)).
\end{proof}
For $n=2$ we immediately get the following corollary (cf. Remark \ref{u:1}).
\begin{corollary}
For $n=2$, a map $\phi_A$ is positive if and only if
$$\sqrt{a_{11}a_{22}}+\sqrt{a_{12}a_{21}}\geq 1.$$
\end{corollary}
\begin{remark}
For $n=3$, the above theorem assures that the condition \eqref{c:2} is necessary for positivity of $\Phi_A$.
\end{remark}
\section{Condition \eqref{c:3}}
Now, we are going to discuss necessity of condition \eqref{c:3}. Let us start with the following
	\begin{theorem}
\label{t:aibc}
		Assume that the matrix \eqref{matrA} is of the form
		\begin{equation}
		A=\left(\begin{array}{ccc}
		a_{1} & b & c\\
		c & a_{2} & b\\
		b & c & a_{3}
		\end{array}\right)
		\label{maibc}
		\end{equation}
		for some arbitrary $a_1,a_2,a_3\geq 1$ and $b,c>0$. If the map $\Phi_A$ given by \eqref{PhiA} is positive then
		\begin{equation}
		(a_1a_2a_3)^{1/3}+b+c\geq 2.
		\label{aibc}
		\end{equation}
	\end{theorem}
	\begin{proof}
		Observe that the Choi matrix of the map $\Phi_A$ is of the form
		\begin{equation}
		C_A=\left(
		\begin{array}{ccc|ccc|ccc}
		a_1 & \cdot & \cdot & \cdot & -1 & \cdot & \cdot & \cdot & -1 \\
		\cdot & b & \cdot & \cdot & \cdot & \cdot & \cdot & \cdot & \cdot \\
		\cdot & \cdot & c & \cdot & \cdot & \cdot & \cdot & \cdot & \cdot \\ \hline
		\cdot & \cdot & \cdot & c & \cdot & \cdot & \cdot & \cdot & \cdot \\
		-1 & \cdot & \cdot & \cdot & a_2 & \cdot & \cdot & \cdot & -1 \\
		\cdot & \cdot & \cdot & \cdot & \cdot & b & \cdot & \cdot & \cdot \\ \hline
		\cdot & \cdot & \cdot & \cdot & \cdot & \cdot & b & \cdot & \cdot \\
		\cdot & \cdot & \cdot & \cdot & \cdot & \cdot & \cdot & c & \cdot \\
		-1 & \cdot & \cdot & \cdot & -1 & \cdot & \cdot & \cdot & a_3
		\end{array}
		\right)  
		\end{equation}	
		It follows from positivity of $\Phi_A$ that $C_A$ is block positive, i.e.
		\begin{equation}
		\langle \xi\otimes\eta, C_A\xi\otimes\eta\rangle\geq 0
		\label{block}
		\end{equation}
		for every $\xi,\eta\in\mathbb{C}^3$.
		
		Let 
		\begin{equation}
		\xi=\eta=\left(
		\left(a_1^{-1}a_2^{\phantom{1}}a_3^{\phantom{1}}\right)^{\frac{1}{12}}, 
		\left(a_1^{\phantom{1}}a_2^{-1}a_3^{\phantom{1}}\right)^{\frac{1}{12}}, 
		\left(a_1^{\phantom{1}}a_2^{\phantom{1}}a_3^{-1}\right)^{\frac{1}{12}}
		\right)^\mathrm{T}.
		\end{equation}
		Then
		\begin{equation}
		\xi\otimes\eta=
		\left(\begin{array}{ccc|ccc|ccc}
		\left(a_1^{-1}a_2^{\phantom{1}}a_3^{\phantom{1}}\right)_{\phantom{1}}^\frac{1}{6}, & a_3^\frac{1}{6}, & a_2^\frac{1}{6} & a_3^\frac{1}{6}, & \left(a_1^{\phantom{1}}a_2^{-1}a_3^{\phantom{1}}\right)_{\phantom{1}}^\frac{1}{6}, & a_1^\frac{1}{6} & a_2^\frac{1}{6}, & a_1^\frac{1}{6}, & \left(a_1^{\phantom{1}}a_2^{\phantom{1}}a_3^{-1}\right)_{\phantom{1}}^\frac{1}{6}
		\end{array}
		\right)^\mathrm{T}.
		\end{equation}
		Hence
		\begin{eqnarray}
		\lefteqn{\langle \xi\otimes\eta,C_A\xi\otimes \eta\rangle=}\\
		&=&
		a_1^{\frac{2}{3}}a_2^{\frac{1}{3}}a_3^{\frac{1}{3}}+
		a_1^{\frac{1}{3}}a_2^{\frac{2}{3}}a_3^{\frac{1}{3}}+
		a_1^{\frac{1}{3}}a_2^{\frac{1}{3}}a_3^{\frac{2}{3}}+
		ba_3^\frac{1}{3}+
		ba_1^\frac{1}{3}+
		ba_2^\frac{1}{3}\\
&&{}+ca_2^\frac{1}{3}+
		ca_3^\frac{1}{3}+
		ca_1^\frac{1}{3} 
		-2a_3^\frac{1}{3}-2a_1^\frac{1}{3}-2a_2^\frac{1}{3}\\
		&=&
		\left(a_1^\frac{1}{3}+a_2^\frac{1}{3}+a_3^\frac{1}{3}\right)
		\left(\left(a_1a_2a_3\right)^\frac{1}{3}+b+c-2\right)
		\end{eqnarray}
		Thus, inequality \eqref{block} leads to \eqref{aibc}.
	\end{proof}
	\begin{remark}
		The vector $\xi$ in the above proof can be taken even simpler:
		\begin{equation}
		\xi=\eta=\left(a_1^{-\frac{1}{6}},a_2^{-\frac{1}{6}},a_3^{-\frac{1}{6}}\right)^\mathrm{T}.
		\label{simpxi}
		\end{equation}
	\end{remark}
	\begin{remark}
		For the case $n=2$ one can prove a similar result for $$A=\left(\begin{array}{cc} a_1 & b_1 \\ b_2 & a_2 \end{array}\right)$$
		where $a_i\geq 1$, $b_i>0$ are arbitrary. In this case one takes
		\begin{equation}
		\xi=\left(a_1^{-\frac{1}{4}}b_2^{\frac{1}{4}}, a_2^{-\frac{1}{4}}b_1^{\frac{1}{4}}\right)^\mathrm{T},\qquad
		\eta= \left(a_1^{-\frac{1}{4}}b_2^{-\frac{1}{4}}, a_2^{-\frac{1}{4}}b_1^{-\frac{1}{4}}\right)^\mathrm{T}
		\end{equation}
	\end{remark}
	\begin{theorem}
		Assume that the matrix \eqref{matrA} is of the form
		\begin{equation}
		A=\left(\begin{array}{ccc}
		a_{1} & b_1 & 0\\
		0 & a_{2} & b_2\\
		b_3 & 0 & a_{3}
		\end{array}\right)
		\end{equation}
		for some arbitrary $a_1,a_2,a_3\geq 1$ and $b_1,b_2,b_3>0$. If the map $\Phi_A$ given by \eqref{PhiA} is positive then
		\begin{equation}
		(a_1a_2a_3)^{1/3}+(b_1b_2b_3)^{1/3}\geq 2.
		\label{aibi}
		\end{equation}	
	\end{theorem}
	\begin{proof}
		The same idea as above. Now, we consider the following wectors $\xi$ and $\eta$ (see \eqref{simpxi}):
		\begin{equation}
		\xi=\left(\begin{array}{c}
		a_1^{-\frac{1}{6}}b_1^{-\frac{1}{6}}b_3^{\frac{1}{6}} \\[2mm]
		a_2^{-\frac{1}{6}}b_2^{-\frac{1}{6}}b_1^{\frac{1}{6}} \\[2mm]
		a_3^{-\frac{1}{6}}b_3^{-\frac{1}{6}}b_2^{\frac{1}{6}}
		\end{array}\right),\qquad
		\eta=\left(\begin{array}{c}
		a_1^{-\frac{1}{6}}b_1^{\frac{1}{6}}b_3^{-\frac{1}{6}} \\[2mm]
		a_2^{-\frac{1}{6}}b_2^{\frac{1}{6}}b_1^{-\frac{1}{6}} \\[2mm]
		a_3^{-\frac{1}{6}}b_3^{\frac{1}{6}}b_2^{-\frac{1}{6}}
		\end{array}\right)
		\end{equation}
Direct calculations, like in the proof of Theorem \ref{t:aibc}, lead to the condition \eqref{aibi}.
	\end{proof}
\section{Problem of sufficiency}	
		Let us consider the matrix $A$ given by \eqref{maibc}. We have shown that the following conditions (cf. \eqref{local}, \eqref{aibc})
		$$(a_ia_{i+1})^{1/2}+(bc)^{1/2}\geq 1,\qquad i=1,2,3,$$
		$$(a_1a_2a_3)^{1/3}+b+c\geq 2.$$
		are necessary for positivity of $\Phi_A$. 
		For $a_1=a_2=a_3$ these condition are also sufficient \cite{CKL92}. 
Note that they are special cases of conditions \eqref{c:2} and \eqref{c:3}.
		
However, these conditions are not sufficient for general case. 
\begin{example}
Let us consider $a_1=\frac{1}{2}$, $a_2=1$, $a_3=2$, $b=1$, $c=0$. 
One checks that conditions \eqref{c:2} and \eqref{c:3} are satisfied.
Let 
$$X=\left(\begin{array}{ccc}
2^{2/3} & 2^{1/6} & 2^{1/6} \\
2^{1/6} & 2^{-1/3} & 2^{-1/3} \\
2^{1/6} & 2^{-1/3} & 2^{-1/3}
\end{array}\right).
$$
One can verify that all principal minor of $X$ are nonnegative, hence $X$ is positive definite.
$$\Phi_A\left(X\right)=
		\left(\begin{array}{ccc} 
2^{2/3} & -2^{1/6} & -2^{1/6} \\ -2^{1/6} &2^{2/3} & -2^{-1/3} \\ -2^{1/6} & -2^{-1/3} & 2^{5/3}\end{array}\right) .$$
Since determinant of the last matrix is equal to $-1$,	it is not positive definite. Hence $\Phi_A$ is not positive.
\end{example}

One can show that the phenomena described in the above example is of a general nature. Namely, we have the following no-go result.
\begin{prop}
Let $a_1,a_2,a_3>0$ and $b\geq 0$ be such that
$$(a_1a_2a_3)^{1/3}+b=2.$$
Then $\Phi_A$ is positive if and only if $a_1=a_2=a_3$.
\end{prop}
\begin{proof}
Sufficiency part follows from Theorem \ref{TH-korea}. We will show necessity. Assume that $\Phi_A$ is a positive map.
Let $$\xi=\left( a_1^{-\frac{1}{6}}a_3^{\frac{1}{6}},\;a_2^{-\frac{1}{6}}a_1^{\frac{1}{6}},\; a_3^{-\frac{1}{6}}a_2^{\frac{1}{6}}\right)^\mathrm{T}.$$
Then
$$\xi\xi^*=\left(\begin{array}{ccc}
a_1^{-\frac{1}{3}}a_3^{\frac{1}{3}} & a_2^{-\frac{1}{6}}a_3^{\frac{1}{6}} & a_1^{-\frac{1}{6}}a_2^{\frac{1}{6}} \\[2mm]
a_2^{-\frac{1}{6}}a_3^{\frac{1}{6}} & a_2^{-\frac{1}{3}}a_1^{\frac{1}{3}} & a_3^{-\frac{1}{6}}a_1^{\frac{1}{6}} \\[2mm] a_1^{-\frac{1}{6}}a_2^{\frac{1}{6}} & a_3^{-\frac{1}{6}}a_1^{\frac{1}{6}} & a_3^{-\frac{1}{3}}a_2^{\frac{1}{3}} 
\end{array}\right)$$
and
$$
\Phi_A(\xi\xi^*)=\left(\begin{array}{ccc}
a_1^{\frac{2}{3}}a_3^{\frac{1}{3}} + ba_2^{-\frac{1}{3}}a_1^{\frac{1}{3}} & -a_2^{-\frac{1}{6}}a_3^{\frac{1}{6}} & -a_1^{-\frac{1}{6}}a_2^{\frac{1}{6}} \\[2mm]
-a_2^{-\frac{1}{6}}a_3^{\frac{1}{6}} & a_2^{\frac{2}{3}}a_1^{\frac{1}{3}} + ba_3^{-\frac{1}{3}}a_2^{\frac{1}{3}} & -a_3^{-\frac{1}{6}}a_1^{\frac{1}{6}} \\[2mm] -a_1^{-\frac{1}{6}}a_2^{\frac{1}{6}} & -a_3^{-\frac{1}{6}}a_1^{\frac{1}{6}} & a_3^{\frac{2}{3}}a_2^{\frac{1}{3}} + ba_1^{-\frac{1}{3}}a_3^{\frac{1}{3}} 
\end{array}\right).
$$
Calculate the determinant of the above matrix
\begin{eqnarray*}
\lefteqn{\det \Phi_A(\xi\xi^*)=}\\
&=& \left(a_1^{\frac{2}{3}}a_3^{\frac{1}{3}} + ba_2^{-\frac{1}{3}}a_1^{\frac{1}{3}}\right)\left(a_2^{\frac{2}{3}}a_1^{\frac{1}{3}} + ba_3^{-\frac{1}{3}}a_2^{\frac{1}{3}}\right)\left(a_3^{\frac{2}{3}}a_2^{\frac{1}{3}} + ba_1^{-\frac{1}{3}}a_3^{\frac{1}{3}}\right) - 2 \\
&&{}-a_1^{-\frac{1}{3}}a_2^{\frac{1}{3}}\left(a_2^{\frac{2}{3}}a_1^{\frac{1}{3}} + ba_3^{-\frac{1}{3}}a_2^{\frac{1}{3}}\right) -a_3^{-\frac{1}{3}}a_1^{\frac{1}{3}}\left(a_1^{\frac{2}{3}}a_3^{\frac{1}{3}} + ba_2^{-\frac{1}{3}}a_1^{\frac{1}{3}}\right) -a_2^{-\frac{1}{3}}a_3^{\frac{1}{3}}\left(a_3^{\frac{2}{3}}a_2^{\frac{1}{3}} + ba_1^{-\frac{1}{3}}a_3^{\frac{1}{3}}\right) \\
&=&
a_1a_2a_3 +3ba_1^{\frac{2}{3}}a_2^{\frac{2}{3}}a_3^{\frac{2}{3}}+3b^2a_1^{\frac{1}{3}}a_2^{\frac{1}{3}}a_3^{\frac{1}{3}}+b^3-2
\\
&&
{}-a_2-ba_1^{-\frac{1}{3}}a_2^{\frac{2}{3}}a_3^{-\frac{1}{3}} -a_1-ba_1^{\frac{2}{3}}a_2^{-\frac{1}{3}}a_3^{-\frac{1}{3}} -a_3-ba_1^{-\frac{1}{3}}a_2^{-\frac{1}{3}}a_3^{\frac{2}{3}} \\
&=&
(a_*+b)^3-2-3\overline{a}-3b\overline{a}a_*^{-1}\\
&=& a_*^{-1}\left(6a_*-3a_*\overline{a}-3b\overline{a}\right) =6a_*^{-1}\left(a_*-\overline{a}\right)
\end{eqnarray*}
In the last line the assumption that $a_*+b=2$ was applied. It follows from positivity of $\Phi_A$ that $a_*\geq \overline{a}$. Hence $a_*=\overline{a}$ and consequently all $a_i$ are equal.
\end{proof}
\begin{remark}
The above considerations that \eqref{c:2} and \eqref{c:3} do not form a set of sufficient conditions for positivity of the map $\Phi_A$ in the case when $a_i$ are distinct numbers. At this stage it still an open problem how to generalize Theorems \ref{TH-korea} and \ref{TH-ccc} for arbitrary numbers $a_i$, $b_j$, $c_k$.
\end{remark} 
\section*{Acknowledgements}
DC  was supported by the Polish National Science Centre project 2015/19/B/ST1/03095, while MM was supported by Templeton Foundation Project "Quantum objectivity: between the whole and the parts".

\bibliographystyle{abbrv}
\bibliography{positive}

\begin{thebibliography}{10}

\bibitem{CKL92}
S.-J. Cho, S.-H. Kye, and S.~G. Lee.
\newblock Generalized choi maps in 3-dimensional matrix algebras.
\newblock {\em Lin. Alg. Appl.}, 171:213–--224, 1992.

\bibitem{Choi75}
M.-D. Choi.
\newblock Completely positive linear maps on complex matrices.
\newblock {\em Lin. Alg. Appl.}, 10(3):285--290, 1975.

\bibitem{Choi75b}
M.-D. Choi.
\newblock Positive semidefinite biquadratic forms.
\newblock {\em Lin. Alg. Appl.}, 12:95--100, 1975.

\bibitem{ChoiLam77}
M.~D. Choi and T.~Y. Lam.
\newblock Extremal positive semidefinite forms.
\newblock {\em Math. Ann.}, 231:1--18, 1977.

\bibitem{ChruSarb2012b}
D.~Chru{\'s}ci{\'n}ski and G.~Sarbicki.
\newblock Exposed positive maps in {$M_4(C)$}.
\newblock {\em Open Sys. Inf. Dyn.}, 19:1250017, 2012.

\bibitem{ChruSarb2014}
D.~Chru{\'s}ci{\'n}ski and G.~Sarbicki.
\newblock Entanglement witnesses: construction, analysis and classification.
\newblock {\em J. Phys. A: Math. Theor.}, 47:483001, 2014.

\bibitem{Ha98}
K.-C. Ha.
\newblock Atomic positive linear maps in matrix algebras.
\newblock {\em Publ. RIMS. Kyoto Univ.}, 134:199, 1998.

\bibitem{Ha02}
K.-C. Ha.
\newblock Positive projections onto spin factors.
\newblock {\em Linear Algebr. Appl.}, 348:105, 2002.

\bibitem{Ha03}
K.-C. Ha.
\newblock A class of atomic positive linear maps in matrix algebras.
\newblock {\em Lin. Alg. Appl.}, 359:277--290, 2003.

\bibitem{HHH96}
M.~Horodecki, P.~Horodecki, and R.~Horodecki.
\newblock Separability of mixed states: necessary and sufficiant conditions.
\newblock {\em Phys. Lett. A}, 223:1--8, 1996.

\bibitem{Jam74}
A.~Jamio{\l}kowski.
\newblock An effective methods of investigation of positive maps on the set of
  positive definite operators.
\newblock {\em Rep. Math. Phys.}, 5:415--424, 1974.

\bibitem{Kye92}
S.-H. Kye.
\newblock A class of atomic positive maps in 3-dimensional matrix algebras.
\newblock In M.~Mathieu, editor, {\em Elementary operators and applications}.
  World Scientific, 1992.

\bibitem{MajMar01}
W.~A. Majewski and M.~Marciniak.
\newblock On a characterization of positive maps.
\newblock {\em J. Phys. A: Math. Gen.}, 34:5863, 2001.

\bibitem{MarRut17}
M.~Marciniak and A.~Rutkowski.
\newblock Merging of positive maps: A construction of various classes of
  positive maps on matrix algebras.
\newblock {\em Lin. Alg. Appl.}, 529:215--257, 2017.

\bibitem{MilOlk15}
M.~Miller and R.~Olkiewicz.
\newblock Stable subspaces of positive maps of matrix algebras.
\newblock {\em Open Syst. Inf. Dyn.}, 22:1550011, 2015.

\bibitem{Osaka91}
H.~Osaka.
\newblock Indecomposable positive maps in low dimensional matrix algebras.
\newblock {\em Lin. Alg. Appl.}, 153:73--83, 1991.

\bibitem{Osaka92}
H.~Osaka.
\newblock A class of extremal positive maps on $3 \times 3$ matrix algebras.
\newblock {\em Publ. RIMS, Kyoto Univ.}, 28:747--756, 1992.

\bibitem{Osaka93}
H.~Osaka.
\newblock A series of absolutely indecomposable positive maps in matrix
  algebras.
\newblock {\em Linear Algebr. Appl.}, 186:45, 1993.

\bibitem{Robertson83}
A.~G. Robertson.
\newblock Schwarz inequalities and decomposition of positive maps on
  c*-algebras.
\newblock {\em Math. Proc. Cambridge Philos. Soc.}, 94:291--296, 1983.

\bibitem{Sto63}
E.~St{\o}rmer.
\newblock Positive linear maps of operator algebras.
\newblock {\em Acta Math.}, 110:233--278, 1963.

\bibitem{Sto82}
E.~St{\o}rmer.
\newblock Decomposable positive maps on {C}*-algebras.
\newblock {\em Proc. Amer. Math. Soc.}, 86:402--404, 1982.

\bibitem{Sto12}
E.~St{\o}rmer.
\newblock {\em Positive linear maps of operator algebras}.
\newblock Springer, 2012.

\bibitem{TanTom88}
K.~Tanabashi and J.~Tomiyama.
\newblock Indecomposable positive maps in matrix algebras.
\newblock {\em Canad. Math. Bull.}, 31:308--317, 1988.

\bibitem{Tang86}
W.-S. Tang.
\newblock On positive linear maps between matrix algebras.
\newblock {\em Linear Algebra Appl.}, 79:33--44, 1986.

\bibitem{Wor76a}
S.~L. Woronowicz.
\newblock Positive maps of low dimensional matrix algebras.
\newblock {\em Rep. Math. Phys.}, 10:165--183, 1976.

\end{thebibliography}

\end{document}